\documentclass[a4paper]{article}
\usepackage{mathrsfs}
\usepackage{latexsym,bm}
\usepackage{amssymb,amsmath,amsthm}
\usepackage[english]{babel}
\usepackage{dsfont}
\usepackage{titlesec}
\usepackage{enumerate}
\usepackage[colorlinks=true]{hyperref}
\usepackage{graphicx}
\usepackage{color}
\usepackage{appendix}
\usepackage{titlesec}
\usepackage{mathrsfs}
\allowdisplaybreaks[4]

\newtheorem{Theorem}{Theorem}[section]
\newtheorem{Lemma}[Theorem]{Lemma}
\newtheorem*{Lemma*}{Lemma}
\newtheorem*{LemmaA}{Lemma A}

\newtheorem{Corollary}[Theorem]{Corollary}

\newtheorem{Remark}[Theorem]{Remark}
\numberwithin{equation}{section}

\hypersetup{linkcolor=blue,urlcolor=red,citecolor=red}

\newcommand{\lc}
{\mathrel{\raise2pt\hbox{${\mathop<\limits_{\raise1pt\hbox
					{\mbox{$\sim$}}}}$}}}

\newcommand{\gc}
{\mathrel{\raise2pt\hbox{${\mathop>\limits_{\raise1pt\hbox{\mbox{$\sim$}}}}$}}}

\newcommand{\ec}
{\mathrel{\raise2pt\hbox{${\mathop=\limits_{\raise1pt\hbox{\mbox{$\sim$}}}}$}}}

\def\bb{\begin{equation}} \def\ee{\end{equation}}

\def\beqn{\begin{eqnarray}}  \def\eqn{\end{eqnarray}}

\def\bbx{\begin{equation*}}   \def\eex{\end{equation*}}

\def\beqnx{\begin{eqnarray*}} \def\eqnx{\end{eqnarray*}}

\def\bd{\begin{description}} \def\ed{\end{description}}

\def \d  {\,\mathrm{d}}

\begin{document}

	\title{Time optimal control for the heat equation with inverse-square potentials }

	\author{
		Le Lu\thanks{ Corresponding author, School of Mathematics and Statistics, Wuhan
			University, Wuhan, 430072, China (2026102010046@whu.edu.cn).}
		\quad \quad
		Ke Wu\thanks{ School of Mathematics and Statistics, Wuhan
			University, Wuhan, 430072, China (kirky\_2023@foxmail.com).}
		\quad \quad
		Can Zhang\thanks{ School of Mathematics and Statistics, Wuhan University, Wuhan, 430072, China (canzhang@whu.edu.cn).}
	}

	\date{}
	\maketitle
	
	\begin{abstract}
		This paper studies the time optimal control problem for the heat equation with singular inverse-square potentials under the Hardy critical condition. We first establish an observability inequality for the singular parabolic equation from general space-time measurable sets of positive Lebesgue measure. Rather than depending on the still-unknown Lebeau-Robbiano spectral inequality for the underlying singular operator, our proof combines Carleman-based observability results over open cylinders, real-analyticity estimates of solutions away from the singular origin, propagation of smallness estimates for real-analytic functions, and a telescoping-series technique. Using this observability result, we derive the null-controllability with controls supported on measurable subsets. Finally, we prove that the corresponding time optimal control is unique and obeys the bang-bang property almost everywhere over the control domain.

	\end{abstract}
	

	\section{Introduction and the main result}
	$\;\;\;\;$ Let $T>0$ and $\Omega\subset\mathbb{R}^n$ ($n\geq3$) be a bounded domain such that $0\in\Omega$, with $C^2$-smooth boundary $\partial\Omega$. The objective of this paper is to study the bang-bang property and the uniqueness of time optimal controls for the following singular heat equation with the inverse-square potentials:
	\begin{equation}\label{eq}
		\begin{cases}
			\partial_t u-\Delta u-\frac{\mu}{|x|^2}u=0,\;\;&\text{in}\;\;\Omega
			\times(0,T),\\
			u=0,\;\;&\text{on}\;\;\partial\Omega\times(0,T),\\
			u(\cdot,0)=u_0,\;\;&\text{in}\;\;\Omega,
		\end{cases}
	\end{equation}
	with $u_0\in L^2(\Omega)$ and $\mu\leq\mu^*\triangleq(n-2)^2/4$.
	The inverse-square potential arises naturally in both quantum mechanics and  combustion theory (see, e.g., \cite{VZ}). It was also proved in \cite{VZ} that  the critical value $\mu^*$ determines the well-posedness of Equation \eqref{eq}. In fact, this problem is closely related to the following classical Hardy inequality
	\begin{equation*}
		\mu^*\int_{\Omega}\frac{u^2}{|x|^2}\,dx
		\leq\int_{\Omega}|\nabla u|^2\,dx,\;\;
		\forall u\in H^1_0(\Omega).
	\end{equation*}
	More precisely, let $H_\mu(\Omega)$, $\mu\leq\mu^*$, be the Hilbert space obtained as the completion of $H_0^1(\Omega)$ with the norm
	$$\|u\|_{H_\mu(\Omega)}\triangleq\Big(\int_\Omega\Big[|\nabla u|^2-\frac{\mu}{|x|^2}u^2\Big]\d x\Big)^{1/2}.$$
	Then for each $\mu\leq \mu^*$, the unbounded operator $A_\mu$ defined by 
	\begin{equation}\label{Amu}
		A_\mu u\triangleq-\Delta u-\frac{\mu}{|x|^2}u,
	\end{equation}
	with the domain $D(A_\mu)\triangleq\big\{u\in H_\mu(\Omega):\;A_\mu u\in L^2(\Omega)\big\},$ generates an analytic semigroup of contractions in $L^2(\Omega)$.
	It is worth pointing out that if $\mu>\mu^*$, then Equation~\eqref{eq} has no solution when $u_0>0$ even locally in time (see \cite{BG} and \cite{VZ}).

	The null-controllability of systems with singular inverse-square potentials has attracted considerable attention. For an interior singularity, it was proved in \cite{VZ1} that Equation~\eqref{eq} can be controlled to zero with a distributed control which surrounds the singularity. Subsequently, the author of \cite{E} removed this restriction and established that Equation~\eqref{eq} can be controlled from any open subset. When the singularity is located on the boundary, it was shown in \cite{Caz} that the null-controllability property is determined by the critical threshold $\mu\leq n^2/4$. The authors of \cite{BZZ} discussed the case of a singular inverse-square potential involving the distance to the boundary. Other related works have also studied the case of variable coefficients (see \cite{XQS}) and multiple distinct singularities (see \cite{XQ}).
	
	The time optimal control problems are also of interest. Such problems for the heat equation with internal and boundary controls were considered in \cite{KW2} and \cite{Micu}, respectively. The authors of \cite{PWZ2} investigated the time optimal control problem for the semilinear heat equation. The time optimal control problem for the Stokes system was studied in \cite{csz}. The work in \cite{wz} considered time optimal control problems for some abstract evolution equations. General frameworks for time optimal control problems of linear evolution equations were developed in \cite{F1} and \cite{wwxz}.
	
	The corresponding time optimal control problem is formulated as follows. Let $\omega\subset\Omega$ be a non-empty open subset. Given any $M>0$, we define the following control constraint set:
	\begin{equation*}
		\mathcal{U}^M=\big\{f\in L^\infty(\Omega\times\mathbb{R}^+):\;
		|f(x,t)|\leq M\;\;\text{for a.e.}\;\;(x,t)\in\Omega\times\mathbb {R}^+\big\}.
	\end{equation*}
	Consider the time optimal control problem
	$$(TP)^M:\;\;T^*(M)\triangleq\inf_{f\in \mathcal{U}^M}\big\{t>0:\;u(\cdot,t;\chi_{\omega}f)
	=0\;\;\text{in}\;\;\Omega\big\},$$
	where $u(\cdot,\cdot\,;\chi_\omega f)$ satisfies the following controlled equation
	\begin{equation*}
		\begin{cases}
			\partial_t u-\Delta u-\frac{\mu}{|x|^2}u=\chi_\omega f,\;\;&\text{in}\;\;
			\Omega\times\mathbb{R}^+,\\
			u=0,\;\;&\text{on}\;\;\partial\Omega\times\mathbb{R}^+,\\
			u(\cdot,0)=u_0,\;\;&\text{in}\;\;
			\Omega,
		\end{cases}
	\end{equation*}
	with $\mu\leq\mu^*$ and $u_0\in L^2(\Omega)$.
	We call $T^*(M)$ the optimal time for $(TP)^M$. A control $f^*\in\mathcal{U}^M$ is called a time optimal control if its state satisfies $u(\cdot, T^*(M); \chi_\omega f^*)=0$. The main result of this paper is the following:
	
	\begin{Theorem}\label{ITP}
		Problem $(TP)^M$ has a unique time optimal control $f^*$. Moreover, it has the bang-bang property: $| f^*(x ,t) | =M\;\;\text{for almost every}\;\;(x,t)\in \omega\times(0,T^*(M)).$
		
	\end{Theorem}

	Our strategy for proving Theorem \ref{ITP} relies crucially on an observability inequality from measurable sets (similar arguments can be found in, e.g., \cite{EZ,KW2,MSE,W}). For parabolic equations, observability inequalities from open cylinders are well understood via Carleman estimates (see, e.g., \cite{FI}, \cite{LR} and the references therein). For the singular heat equation with an inverse-square potential, global Carleman estimates were established in \cite{VZ1} (via spherical harmonics) and later improved in \cite{E} (via a direct weight construction).

	When the observation region is a space-time measurable subset of positive measure, the situation is more delicate. For standard heat equations, observability inequalities from measurable sets were obtained in \cite{PW2} and \cite{AEWZ}. Such observability inequalities were also established for some other equations, for example, the Stokes system (e.g., \cite {csz}) and some abstract evolution equations (e.g., \cite{wz}).

	However, for the singular equation \eqref{eq}, few results are known: the observation region $\omega\times E$ with $\omega$ open and $E$ of positive measure was considered in \cite{ZM}, while  the one-dimensional case with the observation region $\omega\times(0,T)$ (where $\omega$ is of positive measure) was treated in \cite{LL}.
	
	In this paper, we extend the observability inequality on space-time measurable sets from the standard heat equation to the singular equation \eqref{eq} by following \cite{LZ2}. Unlike \cite{AEWZ}, our starting point is the classical open set observability inequality in \cite{E}, since the analogue of the Lebeau-Robbiano inequality for $A_\mu$ with $n\geq3$ remains unknown. The novelty lies in combining the real-analyticity of solutions away from the singularity, propagation of smallness estimates and a telescoping series argument.
	
	The following notations are effective throughout this paper: $C(\cdot)$ denotes a generic positive constant, which depends on the variables in the brackets
	and may vary from line to line; $B_R(x_0)$ stands for the ball
	centered at $x_0$ and of radius $R>0$; $|\omega|$ denotes the Lebesgue measure of the set $\omega$.
	
	The rest of this paper is organized as follows: In Section \ref{observability}, we establish an observability inequality from measurable subsets. In Section \ref{proof_of_bangbang}, we show the $L^{\infty}$-null controllability from measurable subsets, and use the above controllability to prove Theorem~\ref{ITP}. In Section \ref{conclusion}, we give some further remarks.

	\section{Observability inequality from
		measurable subsets}\label{observability}
	
	$\;\;\;\;$In this section, we use Carleman estimates and propagation of smallness from measurable sets for real-analytic functions, together with the telescoping series method, to establish an observability inequality from time-space measurable subsets for Equation \eqref{eq}.
	
	As in \cite{E} and \cite{LC}, we establish the following result without requiring additional geometric restrictions.
	
	\begin{Theorem}\label{obser-D}
		Let $\mu\leq\mu^*$, $T>0$ and $\mathcal{D}\subset\Omega\times(0,T)$ be an $(n+1)$-dimensional Lebesgue measurable subset of positive measure. Then there exists a constant $C=C(\Omega,\mathcal{D},\mu,T)\geq1$ such that
		for any solution $u$ of Equation~\eqref{eq},
		\begin{equation*}
			\Big(\int_\Omega |u(x,T)|^2\d x\Big)^{1/2}\leq C\int_{\mathcal{D}}|u(x,t)|\d x\d t,\;\;\forall\, u_0\in L^2(\Omega).
		\end{equation*}
	\end{Theorem}

	This section consists of four parts. 
	In Subsection \ref{observ_from_open}, we state the observability inequality from open cylinders in \cite{E}. 
	Subsection \ref{analyticity} is devoted to quantifying the real-analyticity estimate for the solutions to Equation~\eqref{eq} in the region away from the singular point. 
	Subsection \ref{interpolation_ineq} presents a global interpolation inequality from measurable subsets.
	Subsection \ref{proof_of_observ} provides the proof of Theorem~\ref{obser-D} via the telescoping series method.
	
	\subsection{Observability inequality from cylinder observation}\label{observ_from_open}
	
	$\;\;\;\;$We use the existing Carleman estimate in
	\cite{E} to establish an observability inequality for any solution to Equation~\eqref{eq}, where the observation is the cylinder $\omega\times(0,T)$ with $\omega\subset\Omega$ being a non-empty open subset.
	
	The following Lemma~\ref{prop2} is a slight modification of
	the observability inequality \cite[(1.7)]{E}, where the dependence of the constant on the time interval $(0,T)$ was not specified. Here, we only state it, leaving its proof to the Appendix.
	\begin{Lemma}\label{prop2}
		There exists a constant $C=C(\Omega,\omega)\geq1$ such that for any $T\in(0,1]$ and any
		solution $w$ to the equation
		\begin{equation*}
			\begin{cases}
				\partial_tw+\Delta w+\frac{\mu}{|x|^2}w=0\;\;&\text{in}\;\;\Omega\times(0,T),\\
				w=0\;\;&\text{on}\;\;\partial\Omega\times(0,T),\\
				w(\cdot,T)=w_T\;\;&\text{in}\;\;\Omega,
			\end{cases}
		\end{equation*}
		where $\mu\leq\mu^*$ and $w_T\in L^2(\Omega)$, satisfies the observability inequality
		\begin{equation}\label{independence}
			\int_{\Omega}|w(x,0)|^2\d x\leq Ce^{\frac{C}{T^6}}\int_{\omega\times(0,T)}|w(x,t)|^2\d x\d t,
			\;\;\forall w_T\in L^2(\Omega).
		\end{equation}
	\end{Lemma}
	
	\begin{Remark}
		The factor $C/T^6$ may not be optimal. However, it is sufficient for our later use in deriving a suitable telescoping series (see Subsection 2.4).
	\end{Remark}

	\subsection{Real-analyticity away from the singularity point}\label{analyticity}
	$\;\;\;\;$The aim of this subsection is to establish the real-analyticity of the solution $u$ to Equation~\eqref{eq} on the subdomain away from the singular point. Since the unbounded operator $A_\mu$ ($\mu\leq\mu^*$, see \eqref{Amu})  is self-adjoint with compact inverse in some suitable functional framework (cf. \cite{VZ}), there exists an orthonormal basis of $L^2(\Omega)$ formed by eigenfunctions of this operator. Thus we have the Fourier series decomposition of any solution $u$ to Equation~\eqref{eq}. Using this decomposition and the classical real-analyticity estimates for solutions to elliptic equations with real-analytic coefficients, we obtain the desired conclusion in the spirit of \cite[Lemma 6]{AEWZ}.

	We are now in a position to state the main result in this subsection. We present its proof in considerable detail, since some of the technical arguments may be useful in establishing similar results for other problems.
	\begin{Lemma}\label{analytic}
		Let $\mu\leq \mu^*$ and $\omega$ be a subdomain of $\Omega$ with $0\notin\overline{\omega}$. Then there are two constants  $C=C(\Omega,\omega,\mu)\geq1$ and $\rho=\rho(\Omega,\omega,\mu)$ with $\rho\in(0,1]$ such that for any solution $u$ to Equation~\eqref{eq},
		\begin{equation*}
			|\partial_x^\alpha\partial_t^\beta u(x,t)|\leq \frac{Ce^{C/(t-s)}|\alpha|!\beta!}{\rho^{|\alpha|}\big((t-s)/2
				\big)^\beta}\|u(\cdot,s)\|_{L^2(\Omega)}, \;\;\forall \alpha\in\mathbb{N}^n,\,\beta\in \mathbb{N},
		\end{equation*}
		when $x\in\omega$ and $0\leq s<t$.
	\end{Lemma}
	\begin{proof}
		It suffices to prove the case that $s=0$.
		For each $\mu\leq\mu^*$, we first assume that $\{w_j^2\}_{j\geq1}$, with
		$$0<w_1\leq w_2\leq\cdots\rightarrow+\infty,$$
		and $\{e_j\}_{j\geq1}$ are accordingly the sets of eigenvalues and $L^2(\Omega)$-normalized eigenfunctions of the operator
		$A_\mu$ with zero Dirichlet boundary condition on $\partial\Omega$, i.e.,
		\begin{equation*}
			\begin{cases}
				-\Delta e_j-\frac{\mu}{|x|^2}e_j=w_j^2e_j,\;\;&\text{in}\;\;\Omega,\\
				e_j=0,\;\;&\text{on}\;\;\partial\Omega.
			\end{cases}
		\end{equation*}
		Let $u$ be the solution of Equation~\eqref{eq} with the initial value $u_0\in L^2(\Omega)$, which
		has the following Fourier series decomposition
		\begin{equation*}
			u_0=\sum_{j\geq1}a_je_j,\;\;\text{where}\;\;\sum_{j\geq1}a_j^2<+\infty.
		\end{equation*}
		Clearly, for each $t>0$ we have that
		\begin{equation*}
			u(x,t)=\sum_{j\geq1}a_je^{-w_j^2t}e_j(x), \;\;x\in\Omega.
		\end{equation*}
		Now, we define a new function $u(\cdot,\cdot,\cdot):\Omega\times\mathbb R\times\mathbb R^+\rightarrow \mathbb R$ by
		\begin{equation*}
			u(x,y,t)=\sum_{j\geq1}a_je^{-w_j^2t+w_jy}e_j(x),\;\;x\in\Omega,
			y\in\mathbb{R}, t>0.
		\end{equation*}
		Then for each $t>0$, it holds that
		\begin{equation*}
			u(x,0,t)=u(x,t),\;\;x\in\Omega,t>0,
		\end{equation*}
		\begin{equation}\label{3-2}
			\partial_t^\beta u(x,y,t)=\sum_{j\geq1}a_j(-w_j^2)^\beta
			e^{-w_j^2t+w_jy}e_j(x),\;\;x\in\Omega, y\in\mathbb{R}, t>0.
		\end{equation}
		Moreover, for each $t>0$, the function $\partial_t^\beta u(\cdot,\cdot,t)$ satisfies
		\begin{equation}\label{a-1}
			\Big(\partial_y^2+\Delta+\frac{\mu}{|x|^2}\Big)\partial_t^\beta
			u(x,y,t)=0\;\;\text{in}\;\;\Omega\times\mathbb{R}.
		\end{equation}
		
		Since $0\notin\overline{\omega}$, without loss of generality
		(otherwise use a finite covering argument), we can assume that there exists a positive constant $R\leq1$ and $x_0\in\omega$ such that
		\begin{equation*}
			\omega\subset B_{R}(x_0),\;\;B_{2R}(x_0)\subset\Omega,\;\;\;\text{and}\;\;\;\;
			B_{2R}(x_0)\cap B_{R}(0)=\emptyset.
		\end{equation*}
		Because  the function $a(x,y)\triangleq\frac{\mu}{|x|^2}$ is real-analytic in $B_{2R}(x_0,0)\subset\Omega\times\mathbb{R}$,
		by the real-analytic estimates of solutions  to linear elliptic
		equations with real-analytic coefficients (cf. \cite[Chapter 3]{J1}), 
		we have that there are constants $C=C(R,|\mu|)\geq1$ and $\rho=\rho(R,|\mu|)\in(0,1)$ such that any solution of Equation~\eqref{a-1} satisfies
		\begin{equation}\label{15-3}
			\|\partial_x^\alpha\partial_t^\beta u(\cdot,\cdot,t)\|_{L^\infty(B_{R}(x_0,0))}\leq \frac{C|\alpha|!}{\rho^{|\alpha|}}
			\Big(\int_{B_{2R}(x_0,0)}|\partial_t^\beta u(x,y,t)|^2\d x\d y
			\Big)^{\frac{1}{2}}\,,\;\;\forall \alpha\in\mathbb{N}^n.
		\end{equation}
		Notice that for each $t>0$,
		\begin{equation}\label{3-3}
			\begin{split}
				\int_{B_{2R}(x_0,0)}|\partial_t^\beta u(x,y,t)|^2\d x\d y
				&\leq \int_{-2R}^{2R}\int_{B_{2R}(x_0)}|\partial_t^\beta u(x,y,t)|^2\d x\d y\\
				&\leq \int_{-2R}^{2R}\int_{\Omega}|\partial_t^\beta u(x,y,t)|^2\d x\d y.
			\end{split}
		\end{equation}
		By the orthonormality of $\{e_j\}_{j\geq1}$ and  the Bessel-Parseval identity, we get from \eqref{3-2} that for each $y\in\mathbb{R}$ and $t>0$,
		\begin{equation}\label{3-4}
			\begin{split}
				\int_{\Omega}|\partial_t^\beta u(x,y,t)|^2\d x
				&=\int_{\Omega}\big|\sum_{j\geq1}a_j(-w_j^2)^\beta
				e^{-w_j^2t+w_jy}e_j(x)\big|^2\d x\\
				&=\sum_{j\geq1}a_j^2w_j^{4\beta}e^{-2w_j^2t+2w_jy}.
			\end{split}
		\end{equation}
		Hence, it follows from \eqref{3-3} and \eqref{3-4} that for each $t>0$
		\begin{equation}\label{15-1}
			\begin{split}
				\int_{B_{2R}(x_0,0)}|\partial_t^\beta u(x,y,t)|^2\d x\d y
				&\leq \sum_{j\geq 1}a_j^2w_j^{4\beta}e^{-2w_j^2t+4Rw_j}\\
				&\leq \max_{j\geq1}\big\{w_j^{4\beta}e^{-w_j^2t}\big\}
				\max_{j\geq 1}\big\{e^{-w_j^2t+4Rw_j}\big\} \sum_{j\geq1}a_j^2.
			\end{split}
		\end{equation}
		For each $\beta\in\mathbb{N}$, note that $\max_{\lambda>0} \{\lambda^{2\beta}e^{-\lambda t}\}=(\frac{2\beta}{te})^{2\beta}$, which, together with the Stirling formula $m!\sim \Big(\frac{m}{e}\Big)^m\sqrt{2\pi m}$, yields that for some generic constant $C\geq1$, it holds that
		\begin{equation}\label{15-2}
			\max_{j\geq1}\big\{w_j^{4\beta}e^{-w_j^2t}\big\}
			\leq C\Big(\frac{2}{t}\Big)^{2\beta}\big(\beta !\big)^2.
		\end{equation}
		Because $\max_{j\geq1}\{ e^{-w_j^2t+4Rw_j} \}\leq e^{4R^2/t},$ it follows from \eqref{15-1} and \eqref{15-2}  that
		\begin{equation*}
			\int_{B_{2R}(x_0,0)}|\partial_t^\beta u(x,y,t)|^2\d x\d y
			\leq Ce^{\frac{4R^2}{t}}\Big(\frac{2}{t}\Big)^{2\beta}\big(\beta !\big)^2.
		\end{equation*}
		This, along with \eqref{15-3}, leads to that
		\begin{equation*}
			\|\partial_x^\alpha\partial_t^\beta u(\cdot,\cdot,t)\|_{L^\infty(B_R(x_0,0))}
			\leq \frac{Ce^{2R^2/t}|\alpha|!\beta!}
			{\rho^{|\alpha|}\big(t/2\big)^\beta}\|u_0\|_{L^2(\Omega)}, \;\;\forall
			\alpha\in\mathbb{N}^n, \,\beta\in\mathbb{N},
		\end{equation*}
		for some positive constants $C=C(R,|\mu|)$ and $\rho=\rho(R,|\mu|)$.
		In particular,
		\begin{equation*}
			\|\partial_x^\alpha\partial_t^\beta u(\cdot,t)\|_{L^\infty(\omega)}
			\leq \frac{Ce^{2R^2/t}|\alpha|!\beta!}
			{\rho^{|\alpha|}\big(t/2\big)^\beta}\|u_0\|_{L^2(\Omega)}, \;\;\forall
			\alpha\in\mathbb{N}^n, \,\beta\in\mathbb{N},
		\end{equation*}
		which completes the proof.
	\end{proof}

	\subsection{Interpolation inequality}\label{interpolation_ineq}
	$\;\;\;\;$The following two lemmas are concerned with the propagation of smallness estimates from measurable sets for real-analytic functions.
	This kind of estimate was introduced in previous works, such as \cite{AEWZ}, \cite{EZ} and \cite{LWYZ}.
	For later use, one lemma is for the one-dimensional case, and the other is for the multi-dimensional case. Their proofs are omitted; we refer the reader to \cite{AE}.
	
	\begin{Lemma}\label{one-d}
		Let $g:[a,a+s]\rightarrow \mathbb{R}$, where $a\in\mathbb{R}$ and $s>0$, be a real-analytic function satisfying
		\begin{equation*}
			\Big|\frac{d^k}{dx^k}g(x)\Big|\leq Mk!(s\rho)^{-k},\;\;\forall x\in[a,a+s],\;\;\forall k\in\mathbb{N},
		\end{equation*}
		for some constants $M>0$ and $\rho\in(0,1]$. Assume that $F\subset[a,a+s]$ is a measurable subset of positive measure.
		Then there are two constants $C=C(\rho,|F|/s)\geq 1$ and $\vartheta
		=\vartheta(\rho,|F|/s)$ with $\vartheta\in(0,1)$ such that
		\begin{equation*}
			\|g\|_{L^\infty(a,a+s)}\leq CM^{1-\vartheta}\Big(\frac{1}{|F|}\int_{F}|g(x)|\d x\Big)^{\vartheta}.
		\end{equation*}
	\end{Lemma}

	\begin{Lemma}\label{multi-d}
		Let $\omega$ be a bounded domain in $\mathbb{R}^n$, $n\geq1$.
		Let $f:\omega\rightarrow \mathbb{R}$ be a real-analytic function
		satisfying
		\begin{equation*}
			|\partial_x^{\alpha}f(x)|\leq M|\alpha|!\rho^{-|\alpha|},\;\;
			\forall x\in\omega,\;\;\forall \alpha\in\mathbb{N}^n,
		\end{equation*}
		for some $M>0$ and $\rho\in(0,1]$.
		Let $\omega_1\subset\omega$ be a measurable subset of positive
		measure. Then there exist two constants $C=C(\omega,\rho,|\omega_1|)\geq1$ and $\vartheta=\vartheta(\omega,\rho,|\omega_1|)$ with $\vartheta\in(0,1)$ such that
		\begin{equation*}
			\|f\|_{L^\infty(\omega)}\leq CM^{1-\vartheta}\Big(
			\int_{\omega_1}|f(x)|\d x\Big)^{\vartheta}.
		\end{equation*}
	\end{Lemma}

	Next, we use Lemmas~\ref{prop2}, \ref{analytic}, \ref{one-d} and \ref{multi-d}, to establish
	an interpolation inequality from measurable sets for any solution
	$u$ to Equation~\eqref{eq}. For similar results, we refer the reader to \cite{AEWZ} and \cite{EZ}.
	\begin{Lemma}\label{interpolation}
		Let $0\leq t_1< t_2\leq 1$, $\eta\in(0,1)$ and $\gamma>0$. Assume that $E\subset(t_1,t_2)$ is a measurable subset, with
		$|E\cap(t_1,t_2)|\geq \eta(t_2-t_1)$, and such that for each $t\in E$, the measurable subset $\mathcal{D}_t\subset\omega$ satisfies $|\mathcal{D}_t|\geq\gamma$. Then there are constants $C=C(\Omega,\omega,|\mu|,\eta,\gamma)\geq1$ and $\vartheta=\vartheta(\Omega,\omega,|\mu|,\eta,\gamma)$ with $\vartheta\in(0,1)$ such that  for any solution $u$ to Equation~\eqref{eq},
		\begin{equation}\label{16-8}
			\|u(\cdot,t_2)\|_{L^2(\Omega)}\leq \Big(\int_{t_1}^{t_2}\chi_{E}(t)\|u(\cdot,t)\|_{L^1(\mathcal{D}_t)}
			\d t\Big)^\vartheta
			\Big(e^{\frac{C}{(t_2-t_1)^6}}\|u(\cdot,t_1)\|
			_{L^2(\Omega)}\Big)^{1-\vartheta}.
		\end{equation}
	\end{Lemma}
	\begin{proof}
		First, we define
		\begin{equation}\label{3-7}
			\tau=t_1+\frac{\eta}{10}(t_2-t_1)
		\end{equation}
		and
		\begin{equation*}
			F=E\cap(\tau,t_2).
		\end{equation*}
		It is clear that $|F|\geq \eta(t_2-t_1)/2$.  From Lemma~\ref{analytic}, noticing \eqref{3-7} we have that
		there are constants $C=C(\Omega,\omega,|\mu|,\eta)\geq1$ and $\rho=\rho(\Omega,\omega,
		|\mu|,\eta)\in(0,1]$ such that for all $t\in[\tau,t_2]$ and for all $x\in\omega$,
		\begin{equation}\label{16-1}
			|\partial_x^\alpha\partial_t^\beta u(x,t)|
			\leq \frac{e^{C/(t_2-t_1)}|\alpha|!\beta!}
			{\rho^{|\alpha|}\big(\eta(t_2-t_1)/20\big)^{\beta}}\|u(\cdot,t_1)
			\|_{L^2(\Omega)},\;\forall \alpha\in\mathbb{N}^n,\;\forall\beta\in 
			\mathbb{N}.
		\end{equation}
		Now, for convenience we write
		\begin{equation}\label{16-2}
			M=e^{\frac{C}{t_2-t_1}}\|u(\cdot,t_1)\|_{L^2(\Omega)}.
		\end{equation}
		Then for each $x\in\omega$, it holds that for all $\beta\in\mathbb{N}$
		\begin{equation*}
			\begin{split}
				|\partial_t^\beta u(x,t)|&\leq M\beta!\big[\eta(t_2-t_1)/20\big]
				^{-\beta}\\
				&\leq M\beta!\big[\eta(t_2-\tau)/20\big]^{-\beta},
				\;\;\forall t\in[\tau,t_2].
			\end{split}
		\end{equation*}
		Hence, it follows from Lemma~\ref{one-d} that there are constants
		$C=C(\eta)\geq1$ and $\vartheta=\vartheta(\eta)$ with $\vartheta\in(0,1)$ such that for any $x\in\omega$,
		\begin{equation} \label{16-5}
			\|u(x,\cdot)\|_{L^\infty(\tau,t_2)}\leq CM^{1-\vartheta}
			\Big(\frac{1}{|F|}\int_{F}|u(x,t)|\d t\Big)^{\vartheta}.
		\end{equation}

		On the other hand, we obtain from Lemma~\ref{prop2} that
		there exists a constant $C=C(\Omega,\omega)\geq1$ such that
		the following observability inequality holds
		\begin{equation}\label{3-9}
			\|u(\cdot,t_2)\|_{L^2(\Omega)}\leq e^{\frac{C}{(t_2-\tau)^6}}
			\|u\|_{L^2(\omega\times(\tau,t_2))}.
		\end{equation}
		By \eqref{16-1} and \eqref{16-2}, we have that
		\begin{equation*}
			\begin{split}
				\|u\|_{L^2(\omega\times(\tau,t_2))}&\leq \|u\|_{L^\infty(\omega\times(\tau,t_2))}^{1/2}
				\|u\|_{L^1(\omega\times(\tau,t_2))}^{1/2}\\
				&\leq e^{\frac{C}{t_2-t_1}}\|u(\cdot,t_1)\|_{L^2(\Omega)}^{1/2}
				\|u\|_{L^1(\omega\times(\tau,t_2))}^{1/2}.
			\end{split}
		\end{equation*}
		This, along with \eqref{3-9}, indicates that
		\begin{equation}\label{16-7}
			\|u(\cdot,t_2)\|_{L^2(\Omega)}\leq \|u\|_{L^1(\omega\times(\tau,t_2))}^{1/2}\Big(e^{\frac{C
				}{(t_2-t_1)^6}}\|u(\cdot,t_1)\|_{L^2(\Omega)}
			\Big)^{1/2}.
		\end{equation}
		By  \eqref{16-5}, it holds that
		\begin{equation*}
			\begin{split}
				\|u\|_{L^1(\omega\times(\tau,t_2))}&\leq (t_2-\tau)\int_{\omega}
				\|u(x,\cdot)\|_{L^\infty(\tau,t_2)}\d x\\
				&\leq CM^{1-\vartheta}\int_\omega\Big(\int_F|u(x,t)|\d t\Big)^{\vartheta}\d x.
			\end{split}
		\end{equation*}
		Thus, we  get from the H\"{o}lder inequality that
		\begin{equation}\label{16-6}
			\|u\|_{L^1(\omega\times(\tau,t_2))}\leq CM^{1-\vartheta}
			\Big(\int_F\int_\omega|u(x,t)|\d x\d t\Big)^\vartheta.
		\end{equation}
		Finally, by \eqref{16-1} again we have that for any $t\in F$,
		\begin{equation*}
			|\partial_x^\alpha u(x,t)|\leq M|\alpha|!\rho^{-|\alpha|},
			\;\;\forall x\in\omega.
		\end{equation*}
		Noticing that $|\mathcal{D}_t|\geq \gamma$ when $t\in F$, we obtain from Lemma~\ref{multi-d} that there are constants $C_1=C_1(\omega,\rho,\gamma)\geq1$ and $\vartheta_1=\vartheta_1(\omega,\rho,\gamma)$ with $\vartheta_1\in(0,1)$
		such that for any $t\in F$,
		\begin{equation*}
			\|u(\cdot,t)\|_{L^\infty(\omega)}\leq C_1M^{1-\vartheta_1}
			\Big(\int_{\mathcal{D}_t}|u(x,t)|\d x\Big)^{\vartheta_1}.
		\end{equation*}
		This, together with \eqref{16-6} and the H\"{o}lder inequality,
		implies that
		\begin{equation*}
			\begin{split}
				\|u\|_{L^1(\omega\times(\tau,t_2))}&\leq C M^{1-\vartheta}
				\Big(\int_F C_1M^{1-\vartheta_1}\big(\int_{\mathcal{D}_t}|u(x,t)|\d x\big)^{\vartheta_1}dt\Big)^{\vartheta}\\
				&\leq CC_1M^{1-\vartheta\vartheta_1}\Big(\int_F\int_{\mathcal{D}_t}|u(x,t)|\d x\d t\Big)^{\vartheta\vartheta_1}.
			\end{split}
		\end{equation*}
		This, combined with \eqref{16-2} and \eqref{16-7}, leads to the desired estimate \eqref{16-8} and completes the proof.
	\end{proof}

	\subsection{Proof of Theorem~\ref{obser-D}}\label{proof_of_observ}
	$\;\;\;\;$Using Fubini's theorem and a property of Lebesgue
	density points for measurable subsets, as well as the interpolation
	inequality established in Lemma~\ref{interpolation}, we can build up a telescoping series for any solution $u$ to Equation~\eqref{eq}. Then the desired observability inequality
	from measurable subsets follows directly from this telescoping series (see also \cite{PW2}, \cite{AEWZ} and \cite{EZ}).
	
	We begin by quoting from \cite[Proposition 2.1]{PW2} the following fact, which is a property of Lebesgue density points and an improved version of \cite[Lemma 2.1.5]{F1}.
	\begin{Lemma}\label{lebesgue-density}
		Let $E\subset(0,T)$ be a measurable subset of positive measure.
		Assume that $\ell\in(0,T)$ is a Lebesgue density point of $E$. Then for each $q\in(0,1)$, there exists a sequence $\{\ell_{m}\}_{m\geq1}\subset(0,T)$,
		which monotone decreasingly converges to $\ell$, such that for any $m\geq1$,
		\begin{equation}\label{16-9}
			\ell_{m+1}-\ell_{m+2}=q(\ell_m-\ell_{m+1})
		\end{equation}
		and
		\begin{equation}\label{16-10}
			|E\cap(\ell_{m+1},\ell_m)|\geq \frac{\ell_m-\ell_{m+1}}{3}.
		\end{equation}
	\end{Lemma}

	Now we are ready to give the proof of Theorem~\ref{obser-D} by the telescoping series method.

	\begin{proof}[\textbf{Proof of Theorem~\ref{obser-D}}]
		Firstly, we can assume, without loss of generality, that $T\in(0,1]$
		and there exists a subdomain $\omega\subset\Omega$ with $0\notin\overline{\omega}$ such that $\mathcal D\subset\omega\times(0,T)$. (Otherwise, we can choose a new measurable subset $\widetilde{\mathcal{D}}\subset\mathcal{D}$ with positive Lebesgue measure and satisfying $\widetilde{\mathcal{D}}\subset\omega_1\times(0,T)$
		with $0\notin\overline{\omega}_1$.) Next, we set for each $t\in(0,T)$ the space slice
		\begin{equation*}
			\mathcal{D}_t=\big\{x\in\omega;\,(x,t)\in\mathcal D\big\}
		\end{equation*}
		and define
		\begin{equation}\label{3-11}
			E=\Big\{
			t\in(0,T);\,|\mathcal{D}_t|\geq \frac{|\mathcal{D}|}{2T}\Big\}.
		\end{equation}
		From Fubini's theorem, we can see that $E$ is measurable with
		$|E|\geq |\mathcal{D}|/(2|\omega|)$ and that
		\begin{equation}\label{16-14}
			\chi_{E}(t)\chi_{\mathcal{D}_t}(x)\leq \chi_{\mathcal D}(x,t),\;\;
			\text{ for a.e.}\; \;(x,t)\in\Omega\times(0,T).
		\end{equation}
		Indeed,
		\begin{equation*}
			\begin{split}
				|\mathcal{D}|&=\int_{0}^T|\mathcal D_t|\d t=\int_E|\mathcal D_t|\d t+\int_{(0,T)\setminus E}|\mathcal D_t|\d t\\
				&\leq |E||\omega|+\frac{|\mathcal D|}{2T}T.
			\end{split}
		\end{equation*}

		Let $\ell\in(0,T)$ be a Lebesgue density point of $E$. Then for each $q\in(0,1)$, which is to be fixed later, by Lemma~\ref{lebesgue-density}, there exists a monotone decreasing
		sequence $\{\ell_m\}_{m\geq1}$ such that \eqref{16-9} and \eqref{16-10} hold, as well as  the following convergence
		\begin{equation}\label{16-11}
			\lim_{m\rightarrow +\infty}\ell_m=\ell.
		\end{equation}
		By \eqref{16-10} and \eqref{3-11}, we can obtain from Lemma~\ref{interpolation} that there are positive constants
		$C=C(\Omega,\omega,\mu,\mathcal D, T)$ and $\vartheta=
		\vartheta(\Omega,\omega,\mu,\mathcal D, T)$ with $\vartheta\in(0,1)$ such that for all $m\geq 1$,
		\begin{equation*}
			\|u(\cdot,\ell_m)\|_{L^2(\Omega)}
			\leq \Big(e^{\frac{C}{(\ell_m-\ell_{m+1})^6}}
			\int_{\ell_{m+1}}^{\ell_m}\chi_{E}\|u(\cdot,t)\|_{L^1(\mathcal D_t)}\d t\Big)^{\vartheta}
			\|u(\cdot,\ell_{m+1})\|_{L^2(\Omega)}^{1-\vartheta}.
		\end{equation*}
		From the Young inequality
		\begin{equation*}
			ab\leq \varepsilon a^p+\varepsilon^{-\frac{r}{p}}b^r,\;\;\text{for all}\;\;a>0, b>0, \varepsilon>0,
		\end{equation*}
		with
		\begin{equation*}
			\frac{1}{p} +\frac{1}{r}=1, p>1,r>1,
		\end{equation*}
		we get that for any $m\geq1$,
		\begin{equation*}
			\|u(\cdot,\ell_m)\|_{L^2(\Omega)}\leq \varepsilon\|u(\cdot,\ell_{m+1})\|_{L^2(\Omega)}
			+\varepsilon^{-\frac{1-\vartheta}{\vartheta}}
			e^{\frac{C}{(\ell_m-\ell_{m+1})^6}}
			\int_{\ell_{m+1}}^{\ell_m}\chi_{E}\|u(\cdot,t)\|_{L^1(\mathcal D_t)}\d t,\;\;\forall \varepsilon>0,
		\end{equation*}
		which is equivalent to that  for any $m\geq1$,
		\begin{equation}\label{16-12}
			\begin{split}
				\varepsilon^{1-\vartheta}e^{-\frac{C}{(\ell_m-\ell_{m+1})^6}}
				\|u(\cdot,\ell_{m})\|_{L^2(\Omega)}
				&-\varepsilon e^{-\frac{C}{(\ell_m-\ell_{m+1})^6}}
				\|u(\cdot,\ell_{m+1})\|_{L^2(\Omega)}\\
				&\leq \int_{\ell_{m+1}}^{\ell_m}\chi_{E}\|u(\cdot,t)\|_{L^1(\mathcal D_t)}\d t,\;\;\forall \varepsilon>0.
			\end{split}
		\end{equation}
		By letting $\varepsilon=e^{-1/(\ell_m-\ell_{m+1})^6}$ in \eqref{16-12}, we have that
		\begin{equation}\label{16-13}
			\begin{split}
				e^{-\frac{C+1-\vartheta}{(\ell_m-\ell_{m+1})^6}}\|u(\cdot,\ell_m)
				\|_{L^2(\Omega)}&-e^{-\frac{C+1}{(\ell_m-\ell_{m+1})^6}}
				\|u(\cdot,\ell_{m+1})\|_{L^2(\Omega)}\\
				&\leq
				\int_{\ell_{m+1}}^{\ell_m}\chi_{E}\|u(\cdot,t)\|_{L^1(\mathcal D_t)}\d t,\;\;\forall m\geq1.
			\end{split}
		\end{equation}
		Finally, we now fix
		\begin{equation*}
			q=\Big(\frac{C+1-\vartheta}{C+1}\Big)^{\frac{1}{6}}, \;\;\text{where $C$ and $\vartheta$
				are two constants given in}\;\;\eqref{16-13}.
		\end{equation*}
		Hence, it follows from \eqref{16-9} and \eqref{16-13} that
		\begin{equation*}
			\begin{split}
				e^{-\frac{C+1-\vartheta}{(\ell_m-\ell_{m+1})^6}}\|u(\cdot,\ell_m)
				\|_{L^2(\Omega)}&-e^{-\frac{C+1-\vartheta}{(\ell_{m+1}-\ell_{m+2})^6}}
				\|u(\cdot,\ell_{m+1})\|_{L^2(\Omega)}\\
				&\leq
				\int_{\ell_{m+1}}^{\ell_m}\chi_{E}\|u(\cdot,t)\|_{L^1(\mathcal D_t)}\d t,\;\;\forall m\geq1.
			\end{split}
		\end{equation*}
		Summing the inequality above from $m=1$ to $+\infty$ (which is called the telescoping series), and using \eqref{16-11} together with the fact that
		\begin{equation*}
			\sup_{t\in(0,T)}\|u(\cdot,t)\|_{L^2(\Omega)}<+\infty,
		\end{equation*}
		we conclude that
		\begin{equation*}
			\|u(\cdot,\ell_{1})\|_{L^2(\Omega)}\leq
			e^{\frac{C+1-\vartheta}{(\ell_1-\ell_2)^6}}
			\int_\ell^{\ell_1}\chi_{E}\|u(\cdot,t)\|_{L^1(\mathcal D_t)}\d t.
		\end{equation*}
		This, together with \eqref{16-14} and the energy decay property $\|u(\cdot,T)\|_{L^2(\Omega)}\leq \|u(\cdot,\ell_{1})\|_{L^2(\Omega)}$, leads to the desired observability inequality from measurable subsets and completes the proof.
	\end{proof}

	\section{Proof of Theorem~\ref{ITP}}\label{proof_of_bangbang}
	$\;\;\;\;$This section consists of two parts. In Subsection \ref{controllable}, we show the $L^\infty$-null controllability from measurable subsets for Equation~\eqref{eq}. In Subsection \ref{bangbang_and_uniqueness}, we use the above controllability to complete the proof of Theorem~\ref{ITP}.

	\subsection{Controllability from measurable sets}\label{controllable}
	$\;\;\;\;$Let us consider the following controlled equation:
	\begin{equation}\label{con-eq1}
		\begin{cases}
			\partial_t u-\Delta u-\frac{\mu}{|x|^2}u=\chi_{\mathcal D}f\;\;&\text{in}\;\;\Omega
			\times(0,T),\\
			u=0\;\;&\text{on}\;\;\partial\Omega\times(0,T),\\
			u(\cdot,0)=u_0\;\;&\text{in}\;\;\Omega,
		\end{cases}
	\end{equation} 
	with $\mu\leq\mu^*$ and $u_0\in L^2(\Omega)$.
	Here $f$ is a control function taken from $L^\infty(\Omega\times(0,T))$. Denote by $u(\cdot,\cdot;\chi_{\mathcal D}f)$ the solution of Equation~\eqref{con-eq1} corresponding to the control $f$.
	The $L^\infty$-null controllability problem is as follows: Given any $u_0\in L^2(\Omega)$, find a control function $f\in L^\infty(\Omega\times(0,T))$
	such that the solution of Equation~\eqref{con-eq1} satisfies
	\begin{equation*}
		u(x,T;\chi_{\mathcal D}f)=0,\;\;\text{for a.e.}\;\;x\in\Omega.
	\end{equation*}

	The $L^2$-null controllability for the controlled equation \eqref{con-eq1} with $\mathcal D=\omega\times(0,T)$, where $\omega\subset\Omega$ is a non-empty open subset,
	was studied in the recent works \cite{VZ1} and \cite{E}. They established that, when $\mu\leq\mu^*$, the equation \eqref{con-eq1} is null-controllable  with the control function $f\in L^2(\Omega\times(0,T))$  being active on the subset $\omega\times(0,T)$.
	
	In the present case, the control function belongs to
	$L^\infty(\Omega\times(0,T))$ and is only active on a subset $\mathcal D\subset\Omega\times(0,T)$ of positive Lebesgue measure. We state the main result of this subsection as follows.
	\begin{Corollary}\label{controllability}
		Let $\mu\leq \mu^*$. Then for any $u_0\in L^2(\Omega)$, there
		exists a control function $f\in L^\infty(\Omega\times(0,T))$,
		with
		\begin{equation*}
			\|f\|_{L^\infty(\Omega\times(0,T))}\leq C(\Omega,\mathcal D,T,\mu)
			\|u_0\|_{L^2(\Omega)},
		\end{equation*}
		such that the solution of Equation~\eqref{con-eq1} satisfies
		\begin{equation}\label{18-4}
			u(x,T;\chi_{\mathcal D}f)=0,\;\;\text{for a.e.}\;\;x\in\Omega.
		\end{equation}
		
	\end{Corollary}
	\begin{proof}
		The proof is somewhat by now standard; for the sake of completeness, we briefly give it here. Let us first define the set $X=\{\chi_{\mathcal D}w\}$ endowed with the usual norm $L^1(\Omega\times(0,T))$, where
		$w$ is any solution of the following equation
		\begin{equation}\label{dual-1}
			\begin{cases}
				\partial_tw+\Delta w+\frac{\mu}{|x|^2}w=0,\;\;&\text{in}\;\;
				\Omega\times(0,T),\\
				w=0,\;\;&\text{on}\;\;\partial\Omega\times(0,T),\\
				w(\cdot,T)=w_T,\;\;&\text{in}\;\;\Omega,
			\end{cases}
		\end{equation}
		with $w_T\in L^2(\Omega)$.
		Clearly, $X$ is a subspace of $L^1(\mathcal D)$.  Then we define the linear mapping $F: X\rightarrow \mathbb R$ by
		\begin{equation*}
			F(\chi_{\mathcal D}w)=-\int_{\Omega}u_0(x)w(x,0)\d x.
		\end{equation*}
		Applying Theorem~\ref{obser-D} to the solution $w$ of Equation~\eqref{dual-1}, we see that $F$ is bounded on $X$. From the Hahn-Banach theorem, there is a linear extension $\widetilde{F}:
		L^1(\mathcal D)\rightarrow \mathbb{R}$ of $F$ such that
		\begin{equation}\label{18-2}
			\widetilde{F}(\chi_{\mathcal D}w)=-\int_{\Omega}u_0(x)w(x,0)\d x,
			\;\;\forall\, \chi_{\mathcal D}w\in X
		\end{equation}
		and
		\begin{equation*}
			|\widetilde{F}(g)|\leq C\|u_0\|_{L^2(\Omega)}\|g\|_{L^1(\mathcal D)},\;\;\forall \,g\in L^1(\mathcal D).
		\end{equation*}
		Thus, $\widetilde{F}$ is a bounded functional on $L^1(\mathcal D)$. By duality, there is $f\in L^\infty(\mathcal D)$ with
		\begin{equation*}
			\|f\|_{L^\infty(\mathcal D)}\leq C\|u_0\|_{L^2(\Omega)},
		\end{equation*}
		such that
		\begin{equation}\label{18-1}
			\widetilde{F}(g)=\int_{\mathcal D}f(x,t)g(x,t)\d x\d t, \;\;\forall \,g\in L^1(\mathcal D).
		\end{equation}
		We extend $f$ over $\Omega\times(0,T)$ by setting it to be zero
		outside $\mathcal D$ and still denote the extended function by $f$. By integration by parts, we have that
		\begin{equation*}
			\int_{\Omega}u(x,T;\chi_{\mathcal D}f)w_{T}(x)\d x
			=\int_{\Omega}u_0(x)w(x,0)\d x+\int_{\Omega}f\chi_{\mathcal D}w\d x\d t,\;\;\forall w_T\in L^2(\Omega).
		\end{equation*}
		This, together with \eqref{18-2} and \eqref{18-1}, implies the equality \eqref{18-4}, and  completes the proof.
	\end{proof}

	\subsection{Bang-Bang property and uniqueness}\label{bangbang_and_uniqueness}
	
	$\;\;\;\;\,$Based on the energy decay property and Corollary~\ref{controllability}, we can conclude by the same argument as in \cite{AEWZ} that $(TP)^M$ has optimal controls. The proof of Theorem~\ref{ITP} is divided into two steps: we first establish the bang-bang property, and then use it to prove the uniqueness.\\
	\textbf{Step 1: Bang-Bang property.} Let $T^*\triangleq T^*(M)$ be the optimal time. Suppose $f^*$ is an optimal control and $u^*$ is the corresponding optimal state; then we have
	\begin{equation}
		\begin{cases}
			\partial_t u^* -\Delta u^*-\frac{\mu}{|x|^2}u^*=\chi_\omega f^*\;\;&\text{in}\;\;\Omega\times\mathbb{R}^+,\\
			u^*=0\;\;&\text{on}\;\;\partial\Omega\times\mathbb{R}^+,\\
			u^*(\cdot,0)=u_0\;\;&\text{in}\;\;\Omega,\\
			u^*(\cdot,T^*) = 0 & \text{in}\;\;\Omega.
		\end{cases}
		\label{time optimal state}
	\end{equation}
	
	By contradiction, we assume that there were $\varepsilon > 0$ and a subset $\widetilde{\mathcal{D}} \subset \omega\times(0,T^*)$ of positive measure such that
	\begin{equation}
		|f^*(x,t)| \leq M-\varepsilon\;\;\text{for a.e.}\;\;(x,t)\in\widetilde{\mathcal{D}}.
		\label{time optimal state.1}
	\end{equation}
	
	It suffices to find some $\delta\in(0,T^*)$ and a pair $(w,g)$ with $g\in\mathcal{U}^M$ such that
	\begin{equation}
		\begin{cases}
			\partial_t w -\Delta w-\frac{\mu}{|x|^2}w=\chi_\omega g\;\;&\text{in}\;\;\Omega
			\times(0,T^*-\delta),\\
			w=0\;\;&\text{on}\;\;\partial\Omega\times(0,T^*-\delta),\\
			w(\cdot,0)=u_0\;\;&\text{in}\;\;\Omega,\\
			w(\cdot,T^*-\delta) = 0 & \text{in } \Omega.
		\end{cases}
		\label{new time optimal state}
	\end{equation}\\
	The existence of such a triple $(\delta,w,g)$ clearly contradicts the optimality of $T^*$ and thus implies the bang-bang property of $f^*$.

	To seek such a triple, we first observe that there is a $\delta_0 > 0$ such that $\mathcal{D}=\widetilde{\mathcal{D}}\cap (\omega \times(\delta_0,T^*))$ has positive measure. Then, it follows from \eqref{time optimal state.1} that
	\begin{equation}
		|f^*(x,t)| \leq M-\varepsilon\;\;\text{for a.e.}\;\;(x,t)\in\mathcal{D}.
		\label{eq:3.2.6.1}
	\end{equation}
	Let $\delta\in(0,\delta_0)$, which will be determined later. By solving the equation
	\begin{equation}
		\begin{cases}
			\partial_t \varphi  -\Delta \varphi-\frac{\mu}{|x|^2}\varphi=0\;\;&\text{in}\;\;\Omega
			\times(\delta,\delta_0),\\
			\varphi=0\;\;&\text{on}\;\;\partial\Omega\times(\delta,\delta_0),\\
			\varphi(\cdot,\delta)=u_0-u^*(\cdot,\delta)\;\;&\text{in}\;\;\Omega.
		\end{cases}
		\label{delta-delta_0}
	\end{equation}
	we get that
	\begin{equation}
		\left \| \varphi (\cdot,\delta_0)  \right \| _{L^2(\Omega) } \le e^{C_1(\delta_0-\delta)}\left \| \varphi (\cdot,\delta)  \right \| _{L^2(\Omega) } \le e^{C_1\delta_0}\left \| \varphi(\cdot,\delta)  \right \| _{L^2(\Omega)}.
		\label{delta-delta_0.1}
	\end{equation}
	where $C_1=C_1(\Omega,\mu) > 0$ is independent on $\delta$.
	
	Next, by Corollary~\ref{controllability}, there exists a control $v \in L^\infty(\Omega\times(\delta_0,T^*))$ such that the solution $\psi$ to the equation
	\begin{equation}
		\begin{cases}
			\partial_t \psi -\Delta \psi-\frac{\mu}{|x|^2}\psi=\chi_{\mathcal{D}}v\;\;&\text{in}\;\;\Omega \times(\delta_0, T^*),\\
			\psi=0\;\;&\text{on}\;\;\partial\Omega\times(\delta_0, T^*),\\
			\psi(\cdot,\delta_0)= \varphi (\cdot,\delta_0)  \;\;&\text{in}\;\;\Omega,
		\end{cases}
		\label{delta_0-T}
	\end{equation}
	satisfies $\psi(\cdot,T^*) = 0$ in $\Omega$. Furthermore, it holds that
	\begin{equation*}
		\|v\|_{L^\infty(\Omega\times(\delta_0,T^*))}\leq C \|\varphi (\cdot ,\delta_0)\|_{L^2(\Omega)}.
	\end{equation*}
	for some $C=C(\Omega,\mathcal{D},\delta_0,T^*\mu) >0 $ independent on $\delta$. Combining the above estimate with \eqref{delta-delta_0.1} we get that
	\begin{equation*}
		\|v\|_{L^\infty(\Omega\times(\delta_0,T^*))}\leq \eta\left \| u_0-u^*(\cdot,\delta)  \right \| _{L^2(\Omega)}.
	\end{equation*}
	where $\eta=Ce^{C_1\delta_0}$ is independent on $\delta$.
	
	Moreover, due to the regularity of $u^*$, we have $\lim_{\delta\to 0}\left \| u_0-u^*(\cdot,\delta)  \right \| _{L^2(\Omega)}=0$, so we can take $\delta$ sufficiently small such that $\left \| u_0-u^*(\cdot,\delta)  \right \| _{L^2(\Omega)} \leq \frac{\varepsilon}{\eta}$ and thus
	\begin{equation}
		\|v\|_{L^\infty(\Omega\times(\delta_0,T^*))}\leq \eta\left \| u_0-u^*(\cdot,\delta)  \right \| _{L^2(\Omega)} \leq \varepsilon.
		\label{v epsilon}
	\end{equation}
	
	We define
	\begin{equation*}
		w_\delta(x,t)=
		\begin{cases}
			\varphi(x,t) + u^*(x,t)\quad\quad\text{when}\; x\in\Omega,\,t\in[\delta,\delta_0),\\
			\psi(x,t) + u^*(x,t)\quad\quad\text{when}\; x\in\Omega,\,t\in[\delta_0,T^*].
		\end{cases}
	\end{equation*}
	\begin{equation*}
		g_\delta(x,t)=f^*(x,t)+\chi_{\mathcal{D}}v(x,t)\quad\quad\text{when}\; x\in\Omega,\,t\in[\delta,T^*].
	\end{equation*}
	On one hand, due to \eqref{time optimal state}, \eqref{delta-delta_0} and \eqref{delta_0-T}, $w_\delta$ solves the following equation
	\begin{equation}\label{w_delta}
		\begin{cases}
			\partial_t w_\delta -\Delta w_\delta-\frac{\mu}{|x|^2}w_\delta=\chi_\omega g_\delta\;\;&\text{in}\;\;\Omega
			\times(\delta,T^*),\\
			w_\delta=0\;\;&\text{on}\;\;\partial\Omega\times(\delta,T^*),\\
			w_\delta(\cdot,\delta)=u_0\;\;&\text{in}\;\;\Omega,\\
			w_\delta(\cdot,T^*) = 0 & \text{in } \Omega.
		\end{cases}
	\end{equation}
	On the other hand, due to (\ref{eq:3.2.6.1}), (\ref{v epsilon}) and the fact $f^*\in\mathcal{U}^M$, we have that
	\begin{equation}\label{g_delta}
		\|g_\delta\|_{L^\infty(\Omega\times(\delta,T^*))}\leq M
	\end{equation}
	
	Finally, we take
	\begin{equation}\label{wg}
		w(x,t)=	w_\delta(x,t+\delta),\quad
		g(x,t)= g_\delta(x,t+\delta)\quad\text{when}\; x\in\Omega,\,t\in[0,T^*-\delta].
	\end{equation}
	Combing \eqref{w_delta}--\eqref{wg} we get that $g\in\mathcal{U}^M$ and $w$ solves Equation \eqref{new time optimal state}, which contradicts the optimality of $T^*$ and thus implies the bang-bang property of $f^*$.\\\\
	\textbf{Step 2: Uniqueness.} Suppose $f_1^*$ and $f_2^*$ are both time optimal controls of $(TP)^M$, with their optimal time $T^*$. Due to step 1, they both have the bang-bang property
	\begin{equation}
		|f^*_1(x,t)|=M\quad\text{and}\quad|f^*_2(x,t)|=M\quad\text{a.e.}\,\,(x,t)\in\omega\times(0,T^*).
		\label{1 2 bang}
	\end{equation} 
	We claim that
	\begin{equation}
		f^*_1(x,t)=f^*_2(x,t),\quad \text{a.e.}\,\,(x,t)\in\omega\times(0,T^*).
		\label{1=2}
	\end{equation}
	Otherwise, we can find an $\varepsilon>0$ and a set of positive  measure $\mathcal{Q}\subset\omega\times(0,T^*)$ such that
	\begin{equation}
		| f^*_1(x,t)-f^*_2(x,t) | \geq \varepsilon,\quad\text{a.e.}\,\,(x,t)\in\mathcal{Q}.
		\label{1 n= 2}
	\end{equation}
	Note that $\bar{f}=\frac{1}{2}(f^*_1+f^*_2)$ is also an optimal control of the $(TP)^M$ problem, thus it also has the bang-bang property
	\begin{equation*}
		|\bar{f}(x,t)|=M\quad\text{a.e.}\,\,(x,t)\in\omega\times(0,T^*).
	\end{equation*}
	However, combining (\ref{1 2 bang}), (\ref{1 n= 2}) and using the parallelogram law, we get that
	\begin{equation*}
		\begin{aligned}
			|f^*_1(x,t)+f^*_2(x,t)|^2&=2|f^*_1(x,t)|^2+2|f^*_2(x,t)|^2-|f^*_1(x,t)-f^*_2(x,t)|^2\\
			&\leq 4M-\varepsilon^2\quad\quad\text{a.e.}\,\,(x,t)\in\mathcal{Q}.
		\end{aligned}
	\end{equation*}
	Then we have
	\begin{equation*}
		|\bar{f}(x,t)|=\frac{1}{2}|f^*_1(x,t)+f^*_2(x,t)|=\frac{1}{2}(4M-\varepsilon^2)^{1/2}<M\quad\text{a.e.}\,\,(x,t)\in\mathcal{Q},
	\end{equation*}
	which contradicts the bang-bang property of $\bar{f}$. Thus we prove the claim (\ref{1=2}) and complete the proof.

	\section{Conclusion}\label{conclusion}
	$\;\;\;\;$This paper studies the time optimal control problem for the heat equation with singular inverse-square potentials. Our main results are twofold. First, we establish an observability inequality from time-space measurable subsets of positive measure for the singular heat equation. Second, we show that the time optimal control problem has a unique optimal control, and this control satisfies the bang-bang property.
	
	We end up with some remarks on the limitations of this work and possible future directions.
	\begin{itemize}
		\item All analysis in this study is restricted to the case where the singular point $x=0$ is isolated and the potential is radially symmetric. Extensions to multiple singularities or more general singular potentials remain open.
		\item The bang-bang property is proved for time optimal controls with the control region inside the domain; the corresponding problem for norm optimal controls or for the boundary control region is also of interest.
		\item It would be interesting to investigate whether similar observability inequalities and bang-bang properties hold for other singular parabolic equations, for example, with potentials involving the distance to the boundary or with time-dependent singularities.
		\item The current proof uses a telescoping series method that requires the underlying semigroup to be analytic and the solution to be real-analytic away from the singularity; exploring whether these techniques can be adapted to degenerate parabolic equations or to fractional diffusion equations is a promising direction.
	\end{itemize}
	
	\section*{Appendix. Proof of Lemma 2.2}
	
	$\;\;\;\;$The theory of Carleman estimates is a powerful tool proposed by Carleman \cite{Carl} to prove unique continuation for a two-dimensional elliptic equation. It is also a standard tool in controllability problems. So far, many works based on Carleman estimates have been published for various PDEs (see, e.g., \cite{BZZ,E,FLZ}).

	Without loss of generality, we can assume that $\overline{B}_1(0)\subset\Omega$ and that $\overline{B}_1(0)\cap\overline{\omega}=\emptyset$. Otherwise, we can take a subset $\omega_1\subset\omega$ such that $\overline{B}_r(0)\cap\overline{\omega_1}=\emptyset$ for some $r>0$, and then use a scaling argument.
	
	Let us introduce the following weight function
	\begin{equation}\label{pei-1}
		\sigma(x,t)=s\theta(t)\Big(e^{2\lambda\sup\psi}-\frac{1}{2}|x|^2
		-e^{\lambda\psi(x)}\Big),\;\;(x,t)\in\Omega\times(0,T),
	\end{equation}
	where $\lambda$ and $s$ are two parameters large enough,
	\begin{equation}\label{theta}
		\theta(t)=\Big(\frac{1}{t(T-t)}\Big)^3,\;\;t\in(0,T),
	\end{equation}
	and $\psi$ is a smooth function satisfying
	\begin{equation*}
		\begin{cases}
			\psi(x)=\ln(|x|),\;\;&x\in B_1(0),\\
			\psi(x)=0,\;\;&x\in\partial\Omega,\\
			\psi(x)>0,\;\;&x\in\Omega\setminus\overline{B}_1(0),\\
			|\nabla\psi(x)|>\delta,\;\;&x\in\Omega\setminus
			\overline{\omega_0},
		\end{cases}
	\end{equation*}
	for an open subset $\omega_0\subset\omega$ and some $\delta>0$.
	We refer the reader to \cite{E} for the construction of $\psi$ and  the choice of the weight function $\sigma$.  Clearly, when $\lambda>0$ is large enough, it holds that
	\begin{equation*}
		\begin{cases}
			\sigma(x,t)>0,\;\;(x,t)\in\Omega\times(0,T),\\
			\lim_{t\rightarrow0^+}\sigma(x,t)=+\infty,\;\lim_{t\rightarrow T^-}\sigma(x,t)=+\infty,\;\;x\in\Omega.
		\end{cases}
	\end{equation*}
	Write $\phi(x)=e^{\lambda\psi(x)}$. Blow, we present the proof of Lemma~\ref{prop2}.

	First, we introduce the following lemma which is an immediate consequence of \cite[Theorem 2.1]{E}.
	\begin{LemmaA}
		There are positive constants $C=C(\Omega,\omega)$ and $\lambda_0
		=\lambda_0(\Omega,\omega)$ such that for each $\lambda\geq\lambda_0$, there is $s_0=s_0(\Omega,\omega,\lambda)$ such that for all
		$s\geq s_0$ and for all $T\in(0,1]$, any solution $w$ to the equation
		\begin{equation*}
			\begin{cases}
				\partial_tw+\Delta w+\frac{\mu}{|x|^2}w=0\;\;&\text{in}\;\;\Omega\times(0,T),\\
				w=0\;\;&\text{on}\;\;\partial\Omega\times(0,T),\\
				w(\cdot,T)=w_T\;\;&\text{in}\;\;\Omega,
			\end{cases}
		\end{equation*}
		where $\mu\leq\mu^*$ and $w_T\in L^2(\Omega)$, satisfies the  weighted estimate
		\begin{multline}\label{carleman estimate}
			s\int_{\Omega\times(0,T)}\theta e^{-2\sigma}\frac{w^2}{|x|}\d x\d t\leq C\Big(s\lambda^2\int_{\omega_0\times(0,T)}\theta\phi e^{-2\sigma}|\nabla w|^2\d x\d t\\
			+s^3\lambda^4\int_{\omega_0\times(0,T)}\theta^3\phi^3e^{-2\sigma}
			w^2\d x\d t\Big).
		\end{multline}
	\end{LemmaA}

	\begin{proof}[\textbf{Proof}]
		The detailed proof of this lemma is implicitly contained in the proof of \cite[Theorem 2.1]{E} and is omitted here.
		Comparing with the statement of \cite[Theorem 2.1]{E}, we emphasize that the positive constant $s_0$ could be uniformly chosen when the time interval $(0,T)$ is bounded.
		We just point out the following minor modifications at the beginning of the proof of \cite[Lemma 2.8]{E}:  Recalling the function
		$\theta(t)$ given by \eqref{theta}, we can obtain that there is a generic constant $C\geq1$ (independent of $T$) such that for any $T\in(0,1]$,
		$$
		|\theta\theta'|\leq C\theta^3,\;|\theta'|\leq C\theta^3,\;
		|\theta''|\leq C\theta^{\frac{5}{3}}.
		$$
	\end{proof}

	Then, by using the above lemma, we have
	
	\begin{proof}[\textbf{Proof of Lemma~\ref{prop2}}]
		By fixing the constants $\lambda\geq\lambda_0$ and $s\geq s_0$
		in the Carleman  estimate \eqref{carleman estimate}, we get that
		there is a constant $C=C(\Omega,\omega)\geq1$ such that for any $T\in(0,1]$,
		\begin{equation}\label{o-1}
			\int_{\Omega\times(0,T)}\theta e^{-2\sigma}\frac{w^2}{|x|}\d x\d t
			\leq C\Big(\int_{\omega_0\times(0,T)}\theta\phi e^{-2\sigma}|\nabla w|^2\d x\d t
			+\int_{\omega_0\times(0,T)}\theta^3\phi^3e^{-2\sigma}w^2\d x\d t
			\Big).
		\end{equation}
		Recalling the definitions of $\theta$ and $\sigma$ given by
		\eqref{pei-1} and \eqref{theta}, we observe  that
		\begin{equation*}
			\theta e^{-2\sigma}\frac{1}{|x|}\geq C\theta e^{-C\theta},\;\;(x,t)\in\Omega\times(0,T),
		\end{equation*}
		for some positive constant $C=C(\Omega,\omega)$.
		For each $t\in[T/4,3T/4]$, it holds that
		$$t^{-1}(T-t)^{-1}\in\Big[\frac{4}{T^2},\frac{16}{3T^2}\Big].$$
		Thus, we have that
		\begin{equation}\label{o-2}
			\theta e^{-2\sigma}\frac{1}{|x|}\geq \frac{C}{T^6}e^{-\frac{C}{T^6}},\;\;\text{when}\;\;(x,t)
			\in\Omega\times\big[\frac{T}{4},\frac{3T}{4}\big].
		\end{equation}
		On the other hand, there exists a positive constant $C=C(\Omega,\omega)$ such that when $(x,t)\in\omega_0\times(0,T)$,
		\begin{equation}\label{o-3}
			\theta\phi e^{-2\sigma}\leq  Ce^{-\theta}
		\end{equation}
		and
		\begin{equation}\label{o-4}
			\theta^3\phi^3e^{-2\sigma}\leq C.
		\end{equation}
		Hence, it follows from \eqref{o-1}, \eqref{o-2}--\eqref{o-4} that
		there is $C=C(\Omega,\omega)\geq1$ such that
		\begin{equation*}
			\frac{C}{T^6}e^{-\frac{C}{T^6}}\int_{\Omega\times\big[
				\frac{T}{4},\frac{3T}{4}\big]}w^2\d x\d t\leq
			C\Big(\int_{\omega_0\times(0,T)}e^{-\sigma}|\nabla w|^2\d x\d t
			+\int_{\omega_0\times(0,T)}w^2\d x\d t\Big).
		\end{equation*}
		This, together with a version of  Caccioppoli's inequality (see,
		for instance, \cite[Lemma III]{VZ1})
		\begin{equation*}
			\int_{\omega_0\times(0,T)}e^{-\sigma}|\nabla w|^2\d x\d t
			\leq C\int_{\omega\times(0,T)}w^2 \d x\d t,
		\end{equation*}
		with some positive constant $C=C(\Omega,\omega)$, leads to that
		\begin{equation*}
			\int_{\frac{T}{4}}^{\frac{3}{4}T}\int_{\Omega}w^2\d x\d t
			\leq Ce^{\frac{C}{T^6}}\int_0^T\int_{\omega}w^2\d x\d t.
		\end{equation*}
		This, along with the energy decay property $\|w(\cdot,0)\|_{L^2(\Omega)}\leq \|w(\cdot,t)\|_{L^2(\Omega)},\;\forall t\in[0,T]$, implies the desired observability inequality \eqref{independence},
		and completes the proof.
	\end{proof}

\end{document}